\documentclass[12pt,reqno]{amsart}
\usepackage{amssymb,latexsym,graphicx,amscd}
\usepackage{lscape}
\usepackage[all]{xy}
\usepackage{xcolor}      
\usepackage{epic,eepic}
\usepackage{xspace}
\usepackage{mathrsfs}
\usepackage{setspace}
\usepackage{float}

\usepackage{tikz}
\usepackage{graphicx}
\usetikzlibrary{calc}
\usetikzlibrary{matrix,arrows}
\usepackage{mathtools} 
\tikzset{desc/.style=auto}
\tikzset{diagmat/.style={matrix of math nodes, row sep=3em, column sep=3em, text height=1.5ex, text depth=0.25ex}}
\tikzset{diagpath/.style={->, font=\scriptsize}}
\usepackage{pgfplots}
\pgfplotsset{my style/.append style={axis x line=middle, axis y line=
middle, axis equal }}
\def\centerarc[#1](#2)(#3:#4:#5){ \draw[#1] ($(#2)+({#5*cos(#3)},{#5*sin(#3)})$) arc (#3:#4:#5); }
\usepackage{mdframed}

\usepackage{caption}
\DeclareCaptionLabelFormat{cont}{#1~#2\alph{ContinuedFloat}}
\usepackage{hyperref}
\hypersetup{
    colorlinks,
    linkcolor={blue},
    citecolor={blue},
    urlcolor={blue!80!black}
}

\newtheorem{Thm}{Theorem}[section]
\newtheorem{Lem}[Thm]{Lemma}

\newtheorem{Cor}[Thm]{Corollary}
\newtheorem{Prop}[Thm]{Proposition}

\theoremstyle{definition}           %ChG

\numberwithin{equation}{section}

\newcommand{\Z}{\mathbb{Z}}

\newcommand{\N}{\mathbb{N}}

\newcommand{\df}{\colon}

\newcommand{\cA}{{\mathcal A}}
\newcommand{\cB}{{\mathcal B}}

\newcommand{\cO}{{\mathcal O}}
\newcommand{\cP}{{\mathcal P}}

\newcommand{\cU}{{\mathcal U}}

\newcommand{\bd}{\mathbf{d}}

\newcommand{\md}{\operatorname{mod}}

\newcommand{\gldim}{\operatorname{gl.dim}}

\newcommand{\dimv}{\underline{\dim}}

\newcommand{\tp}{\operatorname{top}}

\newcommand{\Hom}{\operatorname{Hom}}

\newcommand{\Ext}{\operatorname{Ext}}
\newcommand{\ext}{\operatorname{ext}}
\newcommand{\eend}{\operatorname{end}}
\newcommand{\End}{\operatorname{End}}

\newcommand{\Ima}{\operatorname{Im}}
\newcommand{\Ker}{\operatorname{Ker}}
\newcommand{\Coker}{\operatorname{Coker}}

\newcommand{\irr}{\operatorname{Irr}}

\newcommand{\bsm}{\begin{smallmatrix}}
\newcommand{\esm}{\end{smallmatrix}}

\newcommand{\bbm}{\begin{matrix}}
\newcommand{\ebm}{\end{matrix}}

\newcommand{\bbsm}{\left(\begin{smallmatrix}}
\newcommand{\besm}{\end{smallmatrix}\right)}

\newcommand{\BM}{\left(\begin{matrix}}
\newcommand{\EM}{\end{matrix}\right)}

\newcommand{\GL}{\operatorname{GL}}

\begin{document}
%%%%%%%%%%%%%%%%
%%%%%%%%%%%%%%%%

%\today
\date{June 10, 2024}

%%%%%%%%%%%%%%%%%%%%%%%%%%%%%%%%%%%
\title[Semicontinuous maps on module varieties]
{Semicontinuous maps on module varieties}

%%%%%%%%%%%%%%%%%%%%%%%%%%%%%%%%%%%

\author{Christof Gei{\ss}}
\address{Christof Gei{\ss}\newline
Instituto de Matem\'aticas\newline
Universidad Nacional Aut{\'o}noma de M{\'e}xico\newline
Ciudad Universitaria\newline
04510 M{\'e}xico D.F.\newline
M{\'e}xico}
\email{christof.geiss@im.unam.mx}

\author{Daniel Labardini-Fragoso}
\address{Daniel Labardini-Fragoso\newline
Instituto de Matem\'aticas\newline
Universidad Nacional Aut{\'o}noma de M{\'e}xico\newline
Ciudad Universitaria\newline
04510 M{\'e}xico D.F.\newline
M{\'e}xico}
\email{labardini@im.unam.mx}

\author{Jan Schr\"oer}
\address{Jan Schr\"oer\newline
Mathematisches Institut\newline
Universit\"at Bonn\newline
Endenicher Allee 60\newline
53115 Bonn\newline
Germany}
\email{schroer@math.uni-bonn.de}


%16E30 Homological functors on modules (Tor, Ext, etc.) in 
%associative algebras 

%16Gxx Representation theory of associative rings and algebras 

%16G60 Representation type (finite, tame, wild, etc.) of associative algebras

%16G70 Auslander-Reiten sequences (almost split sequences) and %Auslander-Reiten quivers 

%16P10 Finite rings and finite-dimensional associative algebras 

%13F60 Cluster algebras

%14Mxx Special varieties

%14M99 None of the above, but in this section

%%%%%%%%%%%%%%%%%%%%%%%%%%%%%%%%%%%%%%%%%%%%%%%%

\begin{abstract}
We study semicontinuous maps on varieties of modules over finite-dimensional algebras. 
We prove that truncated Euler maps are upper or lower semicontinuous.
This implies that $g$-vectors and $E$-invariants of modules
are upper semicontinuous.
We also discuss inequalities of generic values of some
upper semicontinuous maps.
\end{abstract}

%%%%%%%%%%%%%%%%
\maketitle
%%%%%%%%%%%%%%%%
\setcounter{tocdepth}{1}
\tableofcontents
\parskip2mm

%%%%%%%%%%%%%%%%%%%%%%%%%%%%%%%%%%%%
%%%%%%%%%%%%%%%%%%%%%%%%%%%%%%%%%%%%

\section{Introduction and main results}\label{sec:intro}

%%%%%%%%%%%%%%%%%%%%%%%%%%%%%%%%%%%%
%%%%%%%%%%%%%%%%%%%%%%%%%%%%%%%%%%%%

%%%%%%%%%%%%%%%%%%%%%%%%%%%%%%%%%%%%%%%%%%%%%%%
\subsection{Irreducible components of module varieties}
%%%%%%%%%%%%%%%%%%%%%%%%%%%%%%%%%%%%%%%%%%%%%%%
Let $A$ be a finite-dimensional $K$-algebra, where $K$ is
an algebraically closed field.
Let $\md(A)$ be the category of finite-dimensional
left $A$-modules.

For $d \ge 0$ let $\md(A,d)$ be the affine variety of
$A$-modules with dimension $d$.
The general linear group $G_d := \GL_d(K)$ acts on $\md(A,d)$ by
conjugation.
The orbits correspond to the isomorphism classes of $d$-dimensional
modules.
The orbit of $M \in \md(A,d)$ is denoted by $\cO_M$.
Let
$\irr(A,d)$ be the set of irreducible components of 
$\md(A,d)$, and let $\irr(A)$ be the union of the sets $\irr(A,d)$  where $d$ runs over all dimensions.
We say that $Z \in \irr(A,d)$ contains a dense orbit if there is
some $M \in \md(A,d)$ with $Z = \overline{\cO_M}$.

%%%%%%%%%%%%%%%%%%%%%%%%%%%%%%%%%%%%%%%%%%%%%%
\subsection{Main results}
%%%%%%%%%%%%%%%%%%%%%%%%%%%%%%%%%%%%%%%%%%%%%%
For $M,M' \in \md(A)$ and $i \ge 0$ let
$$
\ext_A^i(M,M') := \dim \Ext_A^i(M,M').
$$
We have $\Ext_A^0(M,M') = \Hom_A(M,M')$.
Let
$\hom_A(M,M') := \ext_A^0(M,M')$.

For $d,d',i \ge 0$ define
\begin{align*}
\ext_A^i(-,?)&\df \md(A,d) \times \md(A,d') \to \Z,\quad
(M,M') \mapsto \ext_A^i(M,M'),
\\
\ext_A^i(-)&\df \md(A,d) \to \Z,\quad
M \mapsto \ext_A^i(M,M).
\end{align*}
It is well known
that the above maps are upper semicontinuous.

\begin{Thm}\label{thm:main1}
For $d,d',t \ge 0$ the map
$$
\eta_t(-,?)\df \md(A,d) \times \md(A,d') \to \Z,\quad
(M,M') \mapsto \sum_{i=0}^t (-1)^i \ext_A^i(M,M')
$$
is 
upper semicontinuous for $t$ even and lower semicontinuous for
$t$ odd.
\end{Thm}

We call $\eta_t(-,?)$ a \emph{truncated Euler map}. 

If $\gldim(A) = t$,
then Theorem~\ref{thm:main1} says that 
$\eta_t(-,?) = \eta_{t+1}(-,?)$ is upper and lower semicontinuous.
This implies the well known result that in this case
the value $\eta_t(M,M')$ depends only on the dimension
vectors of $M$ and $M'$. 
The associated bilinear form $\Z^n \times \Z^n \to \Z$ is the 
\emph{Euler form} of $A$.
(Here $n$ is the number of simple $A$-modules, up to isomorphism.)

For $Z \in \irr(A,d)$ and $Z' \in \irr(A,d')$ let 
$\ext_A^i(Z,Z')$ (resp. $\ext_A^i(Z)$) be the generic value
of $\ext_A^i(-,?)$ (resp. $\ext_A^i(-)$) on $Z \times Z'$ (resp. $Z$).
Thus
\begin{align*}
\ext_A^i(Z,Z') &= \min\{ \ext_A^i(M,M') \mid (M,M') \in Z \times Z' \},
\\
\ext_A^i(Z) &= \min\{ \ext_A^i(M,M) \mid M \in Z \}.
\end{align*}
Let
$\hom_A(Z,Z') := \ext_A^0(Z,Z')$ and
$\eend_A(Z) := \ext_A^0(Z)$.

For $Z \in \irr(A,d)$ and $Z' \in \irr(A,d')$ let $\eta_t(Z,Z')$ be
the generic value of $\eta_t(-,?)$ on $Z \times Z'$.
The upper semicontinuity of $\ext_A^i(-,?)$ implies that
$$
\eta_t(Z,Z') = \sum_{i=0}^t (-1)^i \ext_A^i(Z,Z').
$$
Theorem~\ref{thm:main1} implies that this generic value is in fact the minimal (for $t$ even) or maximal (if $t$ is odd) value of 
$\eta_t(-,?)$ on $Z \times Z'$.

For
$M,M' \in \md(A)$, $i \ge 1$ and $j \ge 0$ we have 
$$
\Ext_A^i(\Omega^j(M),M') \cong \Ext_A^{i+j}(M,M')
$$ 
where $\Omega^j(M)$ is the  $j$-th syzygy module of
$M$.
For the definition of $\Omega^j(-)$ we use minimal projective resolutions.

\begin{Cor}\label{cor1:main1}
For $j \ge 0$ the map
$$
\hom_A(\Omega^j(-),?)\df \md(A,d) \times \md(A,d') \to \Z,\quad
(M,M') \mapsto \hom_A(\Omega^j(M),M')
$$
is upper semicontinuous.
\end{Cor}

Let $S(1),\ldots,S(n)$ be the simple $A$-modules, up to isomorphism.
For $M \in \md(A)$ and $1 \le i \le n$ let
$$
g_i := g_i(M) := - \hom_A(M,S(i)) + \ext_A^1(M,S(i)).
$$
Then $g(M) := (g_1,\ldots,g_n)$ is the \emph{$g$-vector} of $M$.

\begin{Cor}\label{cor2:main1}
For $d \ge 0$ and $1 \le i \le n$
the map 
$$
g_i(-)\df \md(A,d) \to \Z,\quad
M \mapsto g_i(M)
$$
is upper semicontinuous.
\end{Cor}

Corollary~\ref{cor2:main1} is already proved in our abandoned preprint \cite{GLFS20}.

For $M,N \in \md(A)$ let 
$$
E(M,N) := \hom_A(N,\tau(M))
\text{\quad and \quad}
E(M) := E(M,M)
$$ be the
\emph{$E$-invariant} of $(M,N)$ and $M$, respectively.
Here $\tau$ is the Auslander-Reiten translation for $\md(A)$.
(For an introduction to Auslander-Reiten theory we refer to
\cite{ASS06,ARS97,R84}.)

$E$-invariants and
$g$-vectors appear in Derksen, Weyman and Zelevinsky's
seminal work \cite{DWZ10} on the additive categorification of Fomin-Zelevinsky cluster algebras
via Jacobian algebras.
They also feature in $\tau$-tilting theory \cite{AIR14}.

\begin{Cor}\label{cor3:main1}
For $d,d' \ge 0$
the maps 
\begin{align*}
E(-,?)&\df \md(A,d) \times \md(A,d') \to \Z,\quad
(M,M') \mapsto E(M,M'),
\\
E(-)&\df \md(A,d) \to \Z,\quad
M \mapsto E(M)
\end{align*}
are upper semicontinuous.
\end{Cor}

Corollary~\ref{cor3:main1} is closely related to 
\cite[Corollary~3.7]{DF15}.
Note however that the map in \cite{DF15} is defined on another variety.

For $Z \in \irr(A,d)$ and $Z' \in \irr(A,d')$
the generic value of $E(-,?)$ (resp. $E(-)$) on $Z \times Z'$ (resp. $Z$)
is denoted by $E(Z,Z')$ (resp. $E(Z)$). 

For $Z,Z' \in \irr(A)$ it is often crucial to know when
$E(Z,Z') = 0$.
Corollary~\ref{cor3:main1} implies that it is enough to find some
$(M,M') \in Z \times Z'$ with $E(M,M') = 0$.

By upper semicontinuity
the inequalities
$$
\ext_A^i(Z,Z) \le \ext_A^i(Z)
\text{\quad and \quad}
E(Z,Z) \le E(Z)
$$
hold for all $Z \in \irr(A)$ and $i \ge 0$.

\begin{Thm}\label{thm:main2}
For $Z \in \irr(A)$ the following are equivalent:
\begin{itemize}\itemsep2mm

\item[(i)]
$\hom_A(Z,Z) < \eend_A(Z)$;

\item[(ii)]
$\ext_A^1(Z,Z) < \ext_A^1(Z)$;

\item[(iii)]
$E(Z,Z) < E(Z)$;

\item[(iv)] 
$Z$ does not contain a dense orbit.

\end{itemize}
\end{Thm}

If $Z$ does not contain a dense orbit,
the inequality $\ext_A^i(Z,Z) \le \ext_A^i(Z)$ is not necessarily strict for $i \ge 2$.
Some examples can be found in Section~\ref{sec:examples}.
In general, 
it might be interesting to find a necessary and sufficient condition
for $\ext_A^i(Z,Z) < \ext_A^i(Z)$.

Recall that $M \in \md(A)$ is a \emph{brick} if $\End_A(M) \cong K$.
For $Z \in \irr(A)$ let ${\rm brick}(Z) := \{ M \in Z \mid Z \text{ is a brick} \}$.
Then $Z$ is a \emph{brick component} if ${\rm brick}(Z) \not= \varnothing$.
This is the case if and only if $\eend_A(Z) = 1$ if and only if $\overline{{\rm brick}(Z)} = Z$.

\begin{Cor}\label{cor1:main2}
Let $Z \in \irr(A)$ be a brick component which does not contain a dense orbit.
Then $\hom_A(Z,Z) = 0$.
\end{Cor}

Corollary~\ref{cor1:main2} generalizes one part
of \cite[Theorem~3.5]{S92} from hereditary algebras to arbitrary algebras. It can also be extracted from the proof of 
\cite[Theorem~3.8]{PY20}.

For $Z \in \irr(A)$ let
$$
c(Z) := \min\{ \dim(Z) - \dim \cO_M \mid M \in Z \}
$$
be the \emph{number of parameters} of $Z$.
It follows that $c(Z) \le E(Z)$.
Let
$$
\irr^\tau(A) := \{ Z \in \irr(A) \mid c(Z) = E(Z) \}
$$
be the set of 
\emph{generically $\tau$-reduced} components.
These components were introduced and studied in \cite{GLS12} (where they ran under the name \emph{strongly reduced components}).
A parametrization of $\irr^\tau(A)$ via generic $g$-vectors
is due to Plamondon \cite[Theorem~1.2]{P13}.

The following result is useful for studying direct sum decompositions
of generically $\tau$-reduced components for tame algebras,
see Section~\ref{subsec:directsums} for more details. 
It follows from Corollary~\ref{cor1:main2} combined with
some well known results on tame algebras.

\begin{Cor}[{Plamondon, Yurikusa \cite{PY20}}]\label{cor2:main2}
Assume that $A$ is tame.
For each $Z \in \irr^\tau(A)$ we have
$E(Z,Z) = 0$.
\end{Cor}

Corollary~\ref{cor2:main2} appeared already as part of
\cite[Theorem~3.8]{PY20} where it was mistakenly attributed to 
\cite{GLFS22}.

Generically $\tau$-reduced components play a crucial
role in the construction of well behaved bases for Fomin-Zelevinsky cluster algebras.
Given a $2$-acyclic quiver $Q$ and a non-degenerate potential $W$ for $Q$ let
$A = \cP(Q,W)$ be the associated Jacobian algebra and let
$\cA(Q)$ (resp. $\cU(Q)$) be the Fomin-Zelevinsky cluster algebra
(resp. upper cluster cluster algebra) associated with $Q$, see 
\cite{DWZ08,DWZ10} for 
details.
Recall that $\cA(Q) \subseteq \cU(Q)$.
For each $Z \in \irr^\tau(A)$ there is a \emph{generic Caldero-Chapoton function} $C_Z \in \cU(Q)$.
The set $\cB := \{ C_Z \mid Z \in \irr^\tau(A) \}$ often forms a basis
of $\cA(Q)$, see \cite{GLS12,Q19}.
It is proved in \cite{DWZ10} that all cluster monomials of $\cA(Q)$
are contained in $\cB$.
The tame Jacobian algebras were classified in \cite{GLFS16} and include
all Jacobian algebras arising from triangulations of marked surfaces.
In this case, 
Corollary~\ref{cor2:main2} implies that for each $C_Z \in \cB$ we
have $C_Z^m \in \cB$ for all $m \ge 1$.

In \cite{GLFS24} we construct a geometric version of the
Derksen-Weyman-Zelevinsky mutation of modules over
Jacobian algebras and generalize the mutation invariance of
generically $\tau$-reduced components 
(see \cite[Theorem~1.3]{P13}) from finite-dimensional
Jacobian algebras to arbitrary ones.
Our proof of the mutation invariance uses Corollary~\ref{cor3:main1}.

For the inverse Auslander-Reiten translation $\tau^-$ there are
some obvious dual definitions and statements.

%%%%%%%%%%%%%%%%%%%%%%%%%%%%%
\subsection{Organization of the paper}
%%%%%%%%%%%%%%%%%%%%%%%%%%%%%
In Section~\ref{sec:known} we recall some definitions and well known
results.
Section~\ref{subsec:varieties} contains a few facts on varieties of modules and their connected components.
Section~\ref{subsec:semicont} recalls the definition of upper and
lower semicontinuous maps and lists their basic properties.
Section~\ref{subsec:barresolution} contains the definition of
the bar resolution of $A$.
In Section~\ref{subsec:semicontext}
we use the bar resolution to show that
$\ext_A^i(-,?)$ is upper semicontinuous.
Section~\ref{subsec:gvectors} explains the relation between
$g$-vectors and $E$-invariants.
Section~\ref{subsec:directsums}
recalls the decomposition theorems on direct sums of irreducible components of module varieties.
Section~\ref{sec:proofs} contains the proofs of all results mentioned
in the introduction.
Our first main result
Theorem~\ref{thm:main1} is proved in
Section~\ref{subsec:truncated}.
Corollary~\ref{cor1:main1} is proved in Section~\ref{subsec:omega}.
Section~\ref{subsec:gvectorsproof} contains the proofs of
Corollaries~\ref{cor2:main1} and \ref{cor3:main1}.
The proof of the 
second main result Theorem~\ref{thm:main2} and
of Corollary~\ref{cor1:main2} is in
Section~\ref{subsec:inequalities}.
Section~\ref{subsec:tame} contains the proof of
Corollary~\ref{cor2:main2}.
Section~\ref{sec:examples} consists of a collection of examples.

%%%%%%%%%%%%%%%%%%%
\subsection{Conventions}
%%%%%%%%%%%%%%%%%%%
Throughout,
let $K$ be an algebraically closed field.
By $A$ we always denote a finite-dimensional algebra over $K$.
By a \emph{module} we mean a finite-dimensional left $A$-module.
Let $n(A)$ be the number of isomorphism classes of simple $A$-modules. (This number is finite.)

All varieties are defined over $K$.
For a subset $U$ of an affine variety $X$ let 
$\overline{U}$ be the Zariski closure of $U$ in $X$.

For $d \ge 0$ let $M_d(K)$ be the set of $(d \times d)$-matrices with entries in $K$.

Let $\N$ be the set of natural numbers including $0$, and let
$K^\times := K \setminus \{0\}$.

The cardinality of a set $X$ is denoted by $|X|$.

For maps $f\df X \to Y$ and $g\df Y \to Z$, the composition
is denoted by $gf\df X \to Z$.

%%%%%%%%%%%%%%%%%%%%%%%%%%%%%%%%%%%%%%%%%%%%
%%%%%%%%%%%%%%%%%%%%%%%%%%%%%%%%%%%%%%%%%%%%

\section{Known results on varieties of modules and 
semicontinuous maps}\label{sec:known}

%%%%%%%%%%%%%%%%%%%%%%%%%%%%%%%%%%%%%%%%%%%%
%%%%%%%%%%%%%%%%%%%%%%%%%%%%%%%%%%%%%%%%%%%%

%%%%%%%%%%%%%%%%%%%%%%%%%
\subsection{Varieties of modules}\label{subsec:varieties}
%%%%%%%%%%%%%%%%%%%%%%%%%
For $d \ge 0$ let $\md(A,d)$ be the affine variety of
$K$-algebra homomorphisms
$M\df A \to M_d(K)$.
The general linear group $G_d = \GL_d(K)$ acts by conjugation on
$\md(A,d)$.
(For $g \in G_d$ and $M \in \md(A,d)$ let
$gM\df A \to M_d(K)$ be defined by $a \mapsto g^{-1}M(a)g$.)

The \emph{orbit} of $M$ is
$\cO_M := \{ gM \mid g \in G_d \}$.
Its dimension is $\dim G_d - \dim \End_A(M)$.
The orbits in $\md(A,d)$
correspond to the isomorphism classes of
$d$-dimensional $A$-modules.

Let $n = n(A)$, and let
$S(1),\ldots,S(n)$ be the simple $A$-modules, up to isomorphism.
For a simple $A$-module $S$ and
$M \in \md(A)$ let $[M:S]$ be the Jordan-H\"older multiplicity of
$S$ in a (and therefore in all) composition series of $M$.
Then $\dimv(M) := ([M:S(1)],\ldots,[M:S(n)]) \in \N^n$ is the \emph{dimension vector} of $M$.
Let $d_S$ be the $K$-dimension of $S$.

For $\bd = (d_1,\ldots,d_n) \in \N^n$ with
$d = d_1d_{S(1)} + \cdots + d_nd_{S(n)}$ let
$$
\md(A,\bd) := \{ M \in \md(A,d) \mid \dimv(M) = \bd \}.
$$
This is a connected component of $\md(A,d)$, and all connected components of $\md(A,d)$ arise in this way, see \cite[Corollary~1.4]{Ga74}.

Let $\irr(A,d)$ be the set of irreducible components of the variety
$\md(A,d)$, and let $\irr(A)$ be the union of all $\irr(A,d)$ for
$d \ge 0$.

For $Z \in \irr(A)$
let $\dimv(Z) := \dimv(M)$ where $M$ is any module in $Z$,
and let $\dimv_i(Z) := [M:S(i)]$ be the $i$-th entry of $\dimv(Z)$.

%%%%%%%%%%%%%%%%%%%%%%%%%%
\subsection{Semicontinuous maps}\label{subsec:semicont}
%%%%%%%%%%%%%%%%%%%%%%%%%%
Let $X$ be an affine variety.
A map $\eta\df X \to \Z$ is
\emph{upper semicontinuous} (resp.
\emph{lower semicontinuous}) if for each $n \in \Z$ the set
$X_{\leq n} := \{ x \in X \mid \eta(x) \leq n \}$
(resp. 
$X_{\geq n} := \{ x \in X \mid \eta(x) \geq n \}$)
is open in $X$.
Obviously, $\eta$ is upper semicontinuous if and only if
$-\eta$ is lower semicontinuous.

Let $\eta\df X \to \Z$ be upper semicontinuous (resp.
lower semicontinuous), and let $Z$ be an irreducible
component of $X$.
Let $\eta(Z) := \min\{ \eta(x) \mid x \in Z \}$
(resp. $\eta(Z) := \max\{ \eta(x) \mid x \in Z \}$).
Then $\{ x \in Z \mid \eta(x) := \eta(Z) \}$ is a dense
open subset of $Z$.
We call $\eta(Z)$ the \emph{generic value} of $\eta$ on $Z$.

\begin{Lem}\label{lem:semicont1}
For affine varieties $X$ and $Y$ let $\eta\df X \times Y \to \Z$
be upper or lower semicontinuous.
Then for $x \in X$ and $y \in Y$ the maps
\begin{align*}
\eta(x,-)&\df Y \to \Z,\quad y \mapsto \eta(x,y),
\\
\eta(-,y)&\df X \to \Z,\quad x \mapsto \eta(x,y)
\end{align*} 
are
upper or lower semicontinuous, respectively.
\end{Lem}

\begin{Lem}\label{lem:semicont2}
For an affine variety $X$ let $\eta\df X \times X \to \Z$
be upper or lower semicontinuous.
Then the map
$$
\eta(-)\df X \to \Z,\quad x \mapsto \eta(x,x),
$$
is
upper or lower semicontinuous, respectively.
\end{Lem}

Note that the converses of Lemmas~\ref{lem:semicont1} and \ref{lem:semicont2} are usually wrong.

\begin{Lem}\label{lem:semicont3}
For an affine variety $X$ let $X_1,\ldots,X_t$ be the connected components of $X$.
A map $\eta\df X \to \Z$ is 
be upper or lower semicontinuous if and only if
all restrictions $\eta_{X_i}\df X_i \to \Z$ are
upper or lower semicontinuous, respectively.
\end{Lem}

The following lemma is well known.

\begin{Lem}\label{lem:semicont4}
Let $X$ be an affine variety, let $V$ be a finite-dimensional 
$K$-vector space, and let
$Z$ be a closed subset of $X \times V$.
Assume that for each $x \in X$ the set
$$
Z_x := \{ v \in V \mid (x,v) \in Z \}
$$
is a subspace of $V$.
Then 
$$
\eta\df X \to \Z,\quad x \mapsto \dim(Z_x)
$$
is upper semicontinuous. 
\end{Lem}

%%%%%%%%%%%%%%%%%%%%
\subsection{Bar resolution}\label{subsec:barresolution}
%%%%%%%%%%%%%%%%%%%%
In the following all tensor products $\otimes = \otimes_K$ are 
over $K$ if not indicated otherwise.
For $n \geq 1$ let 
$$
A^{\otimes n} := A \otimes \cdots \otimes A
$$
be the $n$-fold tensor product. 

The \emph{standard complex}
$$
\cdots \xrightarrow{\; p_2\;} A^{\otimes 3} \xrightarrow{\; p_1\;} 
A^{\otimes 2} \xrightarrow{\; p_0\;} A \to 0
$$
is defined by
$$
p_i\df a_0 \otimes \cdots \otimes a_{i+1} \mapsto 
\sum_{k=0}^i (-1)^k a_0 \otimes \cdots \otimes a_ka_{k+1} \otimes \cdots \otimes a_{i+1} 
$$
for $i \ge 0$.
The standard complex is a free $A$-$A$-bimodule resolution of $A$.

Let $M \in \md(A)$.
Applying the tensor functor $- \otimes_A M$ to the standard complex gives a projective
resolution
$$
\cdots 
\xrightarrow{d_3^M}
A^{\otimes 3} \otimes M \xrightarrow{d_2^M}
A^{\otimes 2} \otimes M \xrightarrow{d_1^M}
A \otimes M \xrightarrow{d_0^M}
M \to 0
$$
of $M$, where $d_i^M := p_i \otimes_A M$ for $i \ge 0$.
Note that we identified the (left) $A$-modules 
$A \otimes_A M$ and
$M$.
Thus
all tensor products in the resolution of $M$ are over $K$.

%%%%%%%%%%%%%%%%%%%%%%%%%%%%%%%%%%%%%%%
\subsection{Upper semicontinuity of $\ext_A^i(-,?)$}\label{subsec:semicontext}
%%%%%%%%%%%%%%%%%%%%%%%%%%%%%%%%%%%%%%%
We keep the notation from Section~\ref{subsec:barresolution}.
Let $i \ge 1$.
For simplification we define 
\begin{align*}
k_i^{M,N} &:= \dim \Ker(\Hom_A(d_i^M,N)),
\\
c_i^{M,N} &:= \dim \Hom_A(A^{\otimes i} \otimes M,N).
\end{align*}
The value $c_i^{M,N}$ depends only on $\dim(M)$ and $\dim(N)$
since 
\[
A^{\otimes i} \otimes M \cong A^{\dim(A^{\otimes (i-1)})\dim(M)}
\] 
is free projective and therefore
$c_i^{M,N} = \dim(A^{\otimes (i-1)})\dim(M)\dim(N)$.
(We use the convention $A^{\otimes 0} = K$.)
Applying $\Hom_A(-,N)$ to the projective resolution of $M$ we 
have
\[
\Ext_A^i(M,N) = 
\Ker(\Hom_A(d_{i+1}^M,N))/\Ima(\Hom_A(d_i^M, N)) 
\]
and $\Ext_A^0(M,N) = \Ker(\Hom_A(d_1^M,N))$.
Since 
\[
c_i^{M,N} = \dim \Ima(\Hom_A(d_i^M,N)) +
\dim \Ker(\Hom_A(d_i^M,N))
\] 
we get
\begin{align*}
\ext_A^i(M,N) &=
\dim  \Ker(\Hom_A(d_{i+1}^M,N)) - \dim \Ima(\Hom_A(d_i^M,N))
\\
&= k_{i+1}^{M,N} + k_i^{M,N} - c_i^{M,N}.
\end{align*}
and $\ext_A^0(M,N) = k_1^{M,N}$.

The following result is very useful, e.g. we need it for the proof of
Theorem~\ref{thm:main1}.

\begin{Lem}\label{lem:semicont5}
For $d,d' \ge 0$ and $i \ge 1$
the map 
$$
\eta\df \md(A,d) \times \md(A,d') \to \Z,\quad 
(M,M') \mapsto k_i^{M,M'}
$$
is upper semicontinuous.
\end{Lem}

\begin{proof}
Let 
\begin{align*}
X &:= \md(A,d) \times \md(A,d'),
\\
V &:= \Hom_K(A^{\otimes i} \otimes K^d,K^{d'}),
\\
Z &:= \{ ((M,M'),f) \in X \times V \mid f \in \Hom_A(A^{\otimes i} \otimes M,M'),\,
f d_i^M = 0 \}.
\end{align*} 
Then $Z$ is closed in $X \times V$.

For $(M,M') \in X$ let
$Z_{(M,M')} := \{ f \in V \mid ((M,M'),f) \in Z \}$.
This is obviously a subspace of $V$ which is isomorphic to
$\Ker(\Hom_A(d_i^M,M'))$.
Thus we can apply Lemma~\ref{lem:semicont4} and get that
$\eta$ is upper semicontinuous.
\end{proof}

\begin{Cor}
For $d,d',i \ge 0$ the map
$$
\md(A,d) \times \md(A,d') \to \Z,\quad 
(M,M') \mapsto \ext_A^i(M,M')
$$
is upper semicontinuous.
\end{Cor}

\begin{proof}
This follows directly from Lemma~\ref{lem:semicont5} and the considerations at the beginning of Section~\ref{subsec:semicontext}.
\end{proof}

%%%%%%%%%%%%%%%%%%%%%%%%%%%%%%%%%%
\subsection{$g$-vectors and $E$-invariants}\label{subsec:gvectors}
%%%%%%%%%%%%%%%%%%%%%%%%%%%%%%%%%%
Let $n = n(A)$, and let $S(1),\ldots,S(n)$ be the simple $A$-modules,
up to isomorphism.
For $M \in \md(A)$ and $1 \le i \le n$ recall that
$$
g_i(M) = - \hom_A(M,S(i)) + \ext_A^1(M,S(i)).
$$

Let
$P(1),\ldots,P(n)$ be the indecomposable
projective $A$-modules with $\tp(P(i)) \cong S(i)$ for
$1 \le i \le n$.
Note that for $M \in \md(A)$ we have 
$$
\hom_A(P(i),M) = [M:S(i)].
$$
Each finite-dimensional projective $A$-module $P$ is of the
form
$$
P \cong P(1)^{a_1} \oplus \cdots \oplus P(n)^{a_n}
$$
for some uniquely determined $a_1,\ldots,a_n \ge 0$.
For $1 \le i \le n$ let
$$
[P:P(i)] := a_i.
$$

For $M \in \md(A)$ let
$$
P_1 \xrightarrow{p_1} P_0 \xrightarrow{p_0} M \to 0
$$
be a minimal projective presentation of $M$.

The following lemma is straightforward.

\begin{Lem}
For $1 \le i \le n$ we have
$$
g_i(M) = [P_1:P(i)] - [P_0:P(i)].
$$
\end{Lem}

Let $Z \in \irr(A)$.
By the upper semicontinuity of the maps $\hom_A(-,S(i))$ and
$\ext_A^1(-,S(i))$, there is a dense open subset
$U$ of $Z$ such the restriction of $g_i(-)$ is constant on $U$ for all
$1 \le i \le n$.
We denote this generic value by $g_i(Z)$.

For $M, N \in \md(A)$ recall that
$$
E(M,N) = \hom_A(N,\tau(M))
\text{\quad and \quad}
E(M) = E(M,M).
$$

Applying $\Hom_A(-,N)$ to the minimal projective presentation of $M$ (see above)
we get an exact sequence
\[
0 \to \Hom_A(M,N) \xrightarrow{\Hom_A(p_0,N)} 
\Hom_A(P_0,N) \xrightarrow{\Hom_A(p_1,N)} \Hom_A(P_1,N).
\]
By \cite[Chapter IV, Corollary~4.3]{ARS97} we have
\[
\Coker(\Hom_A(p_1,N)) \cong \Hom_A(N,\tau_A(M)).
\]
In particular, we get
\begin{align*}
E(M,N) &= \hom_A(M,N) - \hom_A(P_0,N) + \hom_A(P_1,N)
\\
&= \hom_A(M,N) - \sum_{i=1}^n [P_0:P(i)][N:S(i)] + 
\sum_{i=1}^n [P_1:P(i)][N:S(i)]
\\
&= \hom_A(M,N) + \sum_{i=1}^n g_i(M)[N:S(i)].
\end{align*}

Let $Z,Z' \in \irr(A)$.
By the upper semicontinuity of $\hom_A(-,?)$, there is a dense open subset
$U$ of $Z \times Z'$ (resp. $Z$) such the restriction of $E(-,?)$ (resp. $E(-)$) is constant on $U$.
We denote this generic value by $E(Z,Z')$ (resp. $E(Z)$).

%%%%%%%%%%%%%%%%%%%%%%%%%%%%%%%%%%%%%%
\subsection{Direct sums of irreducible components}\label{subsec:directsums}
%%%%%%%%%%%%%%%%%%%%%%%%%%%%%%%%%%%%%%
For $1 \le i \le t$ let $Z_i \in \irr(A,d_i)$.
Let $d = d_1 + \cdots + d_t$ and define
\begin{align*}
\eta\df G_d \times Z_1 \times \cdots \times Z_t &\to 
\md(A,d)
\\
(g,M_1,\ldots,M_t) &\mapsto g.(M_1 \oplus \cdots \oplus M_t).
\end{align*}
The image of $\eta$ is denoted by $Z_1 \oplus \cdots \oplus Z_t$.
The Zariski closure $\overline{Z_1 \oplus \cdots \oplus Z_t}$ is 
an irreducible closed subset of $\md(A,d)$, but in
general it is not an irreducible component of $\md(A,d)$.

\begin{Thm}[{\cite[Theorem~1.2]{CBS02}}]\label{thm:decomp1}
For $Z_1,\ldots,Z_t \in \irr(A)$ the following are equivalent:
\begin{itemize}\itemsep2mm

\item[(i)]
$\overline{Z_1 \oplus \cdots \oplus Z_t} \in \irr(A)$;

\item[(ii)]
$\ext_A^1(Z_i,Z_j) = 0$ for all $1 \le i,j \le n$ with $i \not= j$.

\end{itemize}
\end{Thm}

For $Z \in \irr(A)$ and $m \ge 1$ let
$Z^m := Z \oplus \cdots \oplus Z$ be the $m$-fold direct sum
of $Z$.
One calls $Z$ \emph{indecomposable} if it contains a dense subset of indecomposable modules.
Each $Z \in \irr(A)$ can be written (in a unique way) as a direct sum
$$
Z = \overline{Z_1^{a_1} \oplus \cdots \oplus Z_t^{a_t}}
$$
where $Z_1,\ldots,Z_t$ are pairwise different indecomposable components
and $a_1,\ldots,a_t \ge 1$.
We refer to \cite{CBS02} for details.
Let $\Sigma(Z) := t$.

For $Z \in \irr(A)$
Voigt's Lemma (see for example \cite[Proposition~1.1]{Ga74}) and the Auslander-Reiten formulas
(see for example \cite[Section~2.4]{R84}) imply that
$$
c(Z) \le \ext_A^1(Z) \le E(Z).
$$
Recall that
$\irr^\tau(Z) := \{ Z \in \irr(A) \mid
c(Z) = E(Z) \}$ are the \emph{generically $\tau$-reduced} components.

Recall that $M \in \md(A)$ is \emph{rigid} (resp. \emph{$\tau$-rigid}) if
$\Ext_A^1(M,M) = 0$ (resp. $\Hom_A(M,\tau_A(M)) = 0$).
By the Auslander-Reiten formulas, any $\tau$-rigid module is
rigid, wheras the converse is wrong in general.
If $M$ is $\tau$-rigid, then $Z = \overline{\cO_M} \in \irr^\tau(A)$.

A systematic study of $\tau$-rigid modules was initiated in \cite{AIR14}.
These modules are important for the additive categorification
of Fomin-Zelevinsky cluster algebras and they also lead to a broad generalization of classical tilting theory.

\begin{Thm}[{\cite[Theorem~1.3]{CLS15}}]\label{thm:decomp2}
For $Z_1,\ldots,Z_t \in \irr^\tau(A)$ the following are equivalent:
\begin{itemize}\itemsep2mm

\item[(i)]
$\overline{Z_1 \oplus \cdots \oplus Z_t} \in \irr^\tau(A)$;

\item[(ii)]
$E(Z_i,Z_j) = 0$ for all $1 \le i,j \le n$ with $i \not= j$.

\end{itemize}
\end{Thm}

The proof of Theorem~\ref{thm:decomp2} relies on Theorem~\ref{thm:decomp1}.

Note that in Theorems~\ref{thm:decomp1} and \ref{thm:decomp2} we do not assume that the components $Z_1,\ldots,Z_t$ are pairwise different.

By the Auslander-Reiten formulas, the condition
$E(Z_i,Z_j) = 0$ implies that $\ext_A^1(Z_i,Z_j) = 0$.

\begin{Thm}[{\cite[Theorem~6.1]{CLS15}}]\label{prop:numberofsummands}
Let $Z \in \irr^\tau(A)$ with $E(Z,Z) = 0$.
Then $\Sigma(Z) \le n(A)$.
\end{Thm}

It is an open problem if $\Sigma(Z) \le n(A)$ for all
$Z \in \irr^\tau(A)$.

%%%%%%%%%%%%%%%%%%%%%%%%%%%%%%%%%%%%%%%
%%%%%%%%%%%%%%%%%%%%%%%%%%%%%%%%%%%%%%%

\section{Proofs of the main results}\label{sec:proofs}

%%%%%%%%%%%%%%%%%%%%%%%%%%%%%%%%%%%%%%%
%%%%%%%%%%%%%%%%%%%%%%%%%%%%%%%%%%%%%%%

%%%%%%%%%%%%%%%%%%%%%%%%%%%%%%%%%%%%%%%
\subsection{Semicontinuity of truncated Euler maps}
\label{subsec:truncated}
%%%%%%%%%%%%%%%%%%%%%%%%%%%%%%%%%%%%%%%
\begin{proof}[Proof of Theorem~\ref{thm:main1}]
Recall that for $t \ge 0$ and $M,M' \in \md(A)$ we defined
$$
\eta_t(M,M') = \sum_{i=0}^t (-1)^i \ext_A^i(M,M').
$$
We use the notation from Section~\ref{subsec:semicontext}.
For convenience, we set $c_0^{M,M'} := 0$ and define
$$
c^{M,M'} := \sum_{i=0}^t (-1)^{i+1} c_i^{M,M'}.
$$
This is a constant if we fix $\dim(M)$ and $\dim(M')$.

We get
$$
\eta_t(M,M') = 
\begin{cases}
k_{t+1}^{M,M'} + c^{M,M'} & \text{if $t$ is even},
\\
-k_{t+1}^{M,M'} + c^{M,M'} & \text{if $t$ is odd}.
\end{cases}
$$
Now Lemma~\ref{lem:semicont5} implies that
$\eta_t(-,?)$ is upper semicontinuous if $t$ is even and lower
semicontinuous if $t$ is odd.
This finishes the proof.
\end{proof}

%%%%%%%%%%%%%%%%%%%%%%%%%%%%%%%%%%%%%%%%%
\subsection{Upper semicontinuity of $\hom_A(\Omega^j(-),?)$}
\label{subsec:omega}
%%%%%%%%%%%%%%%%%%%%%%%%%%%%%%%%%%%%%%%%%
This section contains the proof of Corollary~\ref{cor1:main1}.
For $M \in \md(A)$ let 
$$
\cdots \xrightarrow{p_3} P_2 \xrightarrow{p_2} P_1 
\xrightarrow{p_1} P_0 
\xrightarrow{p_0} M \to 0
$$
be a minimal projective resolution of $M$.
For $j \ge 0$ we have $\Ima(p_j) = \Omega^j(M)$.

For $j \ge 0$ and $M' \in \md(A)$ define
$$
p_j^{M,M'} := \sum_{i=0}^{j-1} (-1)^i\hom_A(P_{j-1-i},M'). 
$$
(For $j = 0$ we have $p_j^{M,M'} = 0$.)

\begin{Lem}\label{lem:projmult4}
For $d,d',j \ge 0$ we have
$$
\hom_A(\Omega^j(M),M') = 
p_j^{M,M'} +
(-1)^j\eta_j(M,M').
$$.
\end{Lem}

\begin{proof}
Keeping the notation from the beginning of this section, there is a
short exact sequence
$$
0 \to \Omega^j(M) \to P_{j-1} \to \Omega^{j-1}(M) \to 0
$$
for each $j \ge 1$.
Applying $\Hom_A(-,M')$ gives an exact sequence
\begin{multline*}
0 \to \Hom_A(\Omega^{j-1}(M),M') \to \Hom_A(P_{j-1},M') \to
\Hom_A(\Omega^j(M),M')\\
\to \Ext_A^1(\Omega^{j-1}(M),M')
\to 0.
\end{multline*}
Keeping in mind that $\Ext_A^1(\Omega^{j-1}(M),M') \cong
\Ext_A^j(M,M')$ we get the formula
$$
\hom_A(\Omega^j(M),M') 
= -\hom_A(\Omega^{j-1}(M),M') + \hom_A(P_{j-1},M') 
+ \ext_A^j(M,M').
$$
Resolving this recursion we get
\begin{align*}
\hom_A(\Omega^j(M),M') &= 
\left(\sum_{i=0}^{j-1} (-1)^i \hom_A(P_{j-1-i},M')\right)
+ (-1)^j\eta_j(M,M')
\\
&= p_j^{M,M'} + (-1)^j\eta_j(M,M').
\end{align*}
\end{proof}

The next lemma is well known.

\begin{Lem}\label{lem:projmult1}
For $1 \le i \le n(A)$ and $j \ge 0$ we have
$$
\ext_A^j(M,S(i)) = \hom_A(P_j,S(i)) = [P_j:P(i)].
$$
\end{Lem}

\begin{Lem}\label{lem:projmult2}
For $j \ge 1$ we have
$$
p_j^{M,M'}
= \sum_{i=0}^{n(A)}
(-1)^{j-1}\eta_{j-1}(M,S(i)) [M':S(i)].
$$
\end{Lem}

\begin{proof}
Recall that $\hom_A(P(i),M') = [M':S(i)]$. Now
the claim follows from Lemma~\ref{lem:projmult1}.
\end{proof}

\begin{Cor}\label{cor:projmult3}
The map
$$
p_j^{-,?}\df \md(A,d) \times \md(A,d') \mapsto \Z,\quad
(M,M') \mapsto p_j^{M,M'}
$$
is upper semicontinuous.
\end{Cor}

\begin{proof}
Combine Lemmas~\ref{lem:semicont3}, \ref{lem:projmult4}, \ref{lem:projmult2}
with Theorem~\ref{thm:main1}.
\end{proof}

\begin{proof}[Proof of Corollary~\ref{cor1:main1}]
Combining
Lemma~\ref{lem:projmult4} with Corollary~\ref{cor:projmult3} and
Theorem~\ref{thm:main1}
we get that the map
$$
\hom_A(\Omega^j(-),?)\df \md(A,d) \times \md(A,d') \to \Z,\quad
(M,M') \mapsto \hom_A(\Omega^j(M),M')
$$
is upper semicontinuous.
\end{proof}

%%%%%%%%%%%%%%%%%%%%%%%%%%%%%%%%%%%%%%%%%
\subsection{Upper semicontinuity of  $g$-vectors and 
$E$-invariants}\label{subsec:gvectorsproof}
%%%%%%%%%%%%%%%%%%%%%%%%%%%%%%%%%%%%%%%%%
We prove Corollaries~\ref{cor2:main1} and \ref{cor3:main1}
(which correspond to the following two propositions).

\begin{Prop}\label{prop:gvectorusc}
For $1 \le i \le n(A)$ the map
$g_i(-)$ is upper semicontinuous.
\end{Prop}

\begin{proof}
By Theorem~\ref{thm:main1}, the truncated Euler map
\begin{align*}
\eta_1(-,?)\df \md(A,d) \times \md(A,d') &\to \Z
\\
(M,M') &\mapsto \hom_A(M,M') - \ext_A^1(M,M')
\end{align*}
is lower semicontinuous.
Using Lemma~\ref{lem:semicont1} we get that
$-\eta_1(-,S(i))$ is upper semicontinuous.
Since 
$$
g_i(M) = - \hom_A(M,S(i)) + \ext_A^1(M,S(i)) = 
-\eta_1(M,S(i)),
$$
the result follows.
\end{proof}

\begin{Prop}\label{prop:Einvariantusc}
$E(-,?)$ is upper semicontinuous.
\end{Prop}

\begin{proof}
We know that 
\[
E(M,N) = \hom_A(M,N) + \sum_{i=1}^{n(A)} g_i(M)[N:S(i)]
\]
for $M,N \in \md(A)$.
The upper semicontinuity of $\hom_A(-,?)$ and
$g_i(-)$ for all $1 \le i \le n(A)$ implies that $E(-,?)$ is
upper semicontinuous.
\end{proof}

%%%%%%%%%%%%%%%%%%%%%%%%%%%%%%%%%%%%%%
\subsection{Inequalities for $\hom_A(-,?)$, $\ext_A^1(-,?)$ and
$E(-,?)$}\label{subsec:inequalities}
%%%%%%%%%%%%%%%%%%%%%%%%%%%%%%%%%%%%%%
This section contains the proofs of
Theorem~\ref{thm:main2} and
Corollary~\ref{cor1:main2}.

For a constructible subset $U$ of a variety $X$ let
${\rm codim}_X(U)$ be the codimension of $U$ in $X$.

\begin{Lem}\label{lem:diagonaldim}
For $Z \in \irr(A)$ we have
$$
{\rm codim}_{Z \times Z}\left(\{ (M,M') \in Z \times Z \mid 
M \cong M' \}\right)
= c(Z).
$$
\end{Lem}

\begin{proof}
Recall that 
$$
c(Z) = \min\{ \dim(Z) - \dim \cO_M \mid M \in Z \}
$$
and that $\dim \cO_M = \dim G_d - \dim \End_A(M)$.
Let
\begin{align*}
f\df Z \times G_d &\to Z \times Z
\\
(M,g) &\mapsto (M,gM).
\end{align*}
This is a morphism of affine varieties with
$\Ima(f) = \{ (M,M') \in Z \times Z \mid M \cong M' \}$.
Clearly, $\Ima(f)$ is irreducible, and
$\{ (M,M') \in \Ima(f) \mid \dim(Z) - \dim \cO_M = c(Z) \}$ is
a dense open subset of $\Ima(f)$.
For $(M,M')$ in this subset, Chevalley's Theorem says that
$$
\dim \Ima(f) =
\dim(Z \times G_d) - \dim f^{-1}(M,M'). 
$$
We have $\dim f^{-1}(M,M') \cong {\rm Aut}_A(M)$.
We get
$$
\dim \Ima(f) = \dim(Z) + \dim G_d - \dim \End_A(M)
= \dim(Z) + \dim \cO_M.
$$
Thus ${\rm codim}_{Z \times Z}(\Ima(f)) = 2\dim(Z) - \dim(Z) 
- \dim \cO_M = c(Z)$.
\end{proof}

\begin{proof}[Proof of Theorem~\ref{thm:main2}]
(i) $\implies$ (ii):
Assume that $\hom_A(Z,Z) < \eend_A(Z)$.
We choose $(M,M') \in Z \times Z$ such that
\begin{align*}
\hom_A(M,M') &= \hom_A(Z,Z), & \ext_A^1(M,M') &= \ext_A^1(Z,Z), 
\\
\eend_A(M) &= \eend_A(Z), &\ext^1_A(M) &= \ext_A^1(Z).
\end{align*}
The short exact sequence
$$
0 \to \Omega(M) \to P_0 \to M \to 0
$$
yields equalities
\begin{align*}
\ext_A^1(M,M') &= 
\hom_A(M,M') - \hom_A(P_0,M') + \hom_A(\Omega(M),M'),
\\
\ext_A^1(M) &= 
\eend_A(M) - \hom_A(P_0,M) + \hom_A(\Omega(M),M).
\end{align*}
The map $\hom_A(P_0,-)$ is constant on $Z$.
By the upper semicontinuity of the map $\hom_A(\Omega(-),?)$ 
(see Corollary~\ref{cor1:main1}) we can assume that
$\hom_A(\Omega(M),M') \le \hom_A(\Omega(M),M)$.
Now $\hom_A(Z,Z) < \eend_A(Z)$ implies $\ext_A^1(Z,Z) < \ext_A^1(Z)$.

(i) $\implies$ (iv),
(ii) $\implies$ (iv) and (iii) $\implies$ (iv):
Assume that $Z \in \irr(A)$ contains a dense orbit, i.e. $Z = \overline{\cO_M}$
for some $M \in Z$.
Then $\cO_M$ and $\cO_M \times \cO_M$ are dense open subsets of $Z$ and
$Z \times Z$, respectively. 
We get $\ext_A^i(Z,Z) = \ext_A^i(Z)$ for all $i \ge 0$ 
and $E(Z,Z) = E(Z)$.

(iv) $\implies$ (i):
Assume that $Z \in \irr(A,d)$ does not contain a dense orbit.
This implies $c(Z) \ge 1$.
We want to show that $\hom_A(Z,Z) < \eend_A(Z)$.

Let 
$$
U := \{ (M,M',f) \in Z \times Z \times \Hom_K(K^d,K^d) \mid
f \in \Hom_A(M,M') \}.
$$
This is a closed subset of
$Z \times Z \times \Hom_K(K^d,K^d)$.
For $i \ge 0$ let
$$
U_i := \{ (M,M',f) \in U \mid \hom_A(M,M') = i \}.
$$
This is a locally closed subset of $U$.
Let $\pi_i'\df U_i \to Z \times Z$ be the obvious projection,
and define $V_i := \Ima(\pi_i')$.
We get a surjective map $\pi_i\df U_i \to V_i$.
By \cite[Lemma~2.1]{B96} the map
$\pi_i\df U_i \to V_i$ is a vector bundle.
In particular, $\pi_i$ maps open subsets of $U_i$ to open subsets
of $V_i$.

Let $i_0 := \hom_A(Z,Z)$.
Then $V_{i_0}$ is a dense open subset of $Z \times Z$ and
therefore irreducible.

Let $M \in Z$ such that $\dim \End_A(M) = \eend_A(Z)$.
We know already that $i_0 = \hom_A(Z,Z) \le \eend_A(Z)$.

Let
$$
U_{i_0}^\circ := \{ (M,M',f) \in U_{i_0} \mid f \text{ is not an isomorphism} \}.
$$
This is a closed subset of $U_{i_0}$.
Thus $W_{i_0} := U_{i_0} \setminus U_{i_0}^\circ$ is open
in $U_{i_0}$.

Suppose that
$W_{i_0} \not= \varnothing$.
Then $\pi_{i_0}(W_{i_0})$ 
is a dense open subset of $V_{i_0}$ and therefore of 
$Z \times Z$.

Note that for all $(M,M') \in \pi_{i_0}(W_{i_0})$ we have $M \cong M'$.
This is a contradiction, since by Lemma~\ref{lem:diagonaldim}
we have
$$
{\rm codim}_{Z \times Z}(\{ (M,M') \in Z \times Z \mid
M \cong M'\}) = c(Z) \ge 1.
$$
Thus we proved that $U_{i_0}^\circ = U_{i_0}$.
In particular, if $(M,M') \in Z \times Z$ such that
$\hom_A(M,M') = i_0$, then $M \not\cong M'$.
This implies
$\hom_A(Z,Z) < \eend_A(Z)$.

(iv) $\implies$ (iii):
Let $Z \in \irr(A)$.
Assume that $Z$ does not contain a dense orbit.
We know already that (iv) implies (i), i.e. 
$\hom_A(Z,Z) < \eend_A(Z)$.
Now the two equations
\begin{align*}
E(Z,Z) &= \hom_A(Z,Z) + \sum_{i=1}^{n(A)} g_i(Z)\dimv_i(Z),
\\
E(Z) &= \eend_A(Z) + \sum_{i=1}^{n(A)} g_i(Z)\dimv_i(Z).
\end{align*}
imply that
$E(Z,Z) < E(Z)$.

This finishes the proof of Theorem~\ref{thm:main2}.
\end{proof}

\begin{proof}[Proof of Corollary~\ref{cor1:main2}]
Let $Z \in \irr(A)$ be a brick component.
By definition
we have $\eend_A(Z) = 1$.
If $Z$ does not have a dense orbit, Theorem~\ref{thm:main2}
implies that $\hom_A(Z,Z) = 0$.
\end{proof}

%%%%%%%%%%%%%%%%%%%%%%%%%%%%%%%%%%%%%%%%%%%%%%
\subsection{Generically $\tau$-reduced components for tame algebras}\label{subsec:tame}
%%%%%%%%%%%%%%%%%%%%%%%%%%%%%%%%%%%%%%%%%%%%%%
We prove Corollary~\ref{cor2:main2}.

For $Z \in \irr(A)$ we have $c(Z) = 0$ if and only if $Z$ contains
a dense orbit.
Furthermore, we have $c(Z) = E(Z) = 0$ if and only if
$Z$ contains a $\tau$-rigid module $M$.
In this case,
$\cO_M$ is dense in $Z$, and we have
$E(Z,Z) = 0$.

From now on assume that $A$ is a tame algebra, and let $Z \in \irr(A)$ be indecomposable.
It follows that $c(Z) \le 1$, see for example \cite[Section~2.2]{CC15}.
In this situation,
the following are equivalent:
\begin{itemize}\itemsep2mm

\item[(i)]
$c(Z) = E(Z) = 1$;

\item[(ii)]
$Z$ is a brick component which does not contain a dense orbit.

\end{itemize}
For more details we refer to \cite[Section~3]{GLFS22}.
In this case,
Theorem~\ref{thm:main2}
says that $E(Z,Z) = 0$.
Now Theorem~\ref{thm:decomp2} implies Corollary~\ref{cor2:main2}.

%%%%%%%%%%%%%%%%%%%%%%%%%%%%%%%%%%%
%%%%%%%%%%%%%%%%%%%%%%%%%%%%%%%%%%%

\section{Examples}\label{sec:examples}

%%%%%%%%%%%%%%%%%%%%%%%%%%%%%%%%%%%
%%%%%%%%%%%%%%%%%%%%%%%%%%%%%%%%%%%

%%%%%%%%%%%%%
\subsection{}
%%%%%%%%%%%%%
Let $Z \in \irr(A)$.
For $i \ge 2$ 
the inequalities $\ext_A^i(Z,Z) \le \ext_A^i(Z)$ are in general not strict.
As a trivial example, if
$\gldim(A) = t$,
then for all $i \ge t+1$ and $Z \in \irr(A)$ we have 
$\ext_A^i(Z,Z) = \ext_A^i(Z) = 0$.

%%%%%%%%%%%%%
\subsection{}
%%%%%%%%%%%%%
Let $A = KQ/I$ where $Q$ is the quiver
$$
\xymatrix{
1 \ar@(ul,dl)[]_a\ar@(ur,dr)^b
}
$$
and $I$ is generated by $\{ ab -ba,\, a^2,\, b^2 \}$.
(If ${\rm char}(K) = 2$, then this is the group algebra of the Kleinean four group.)
For $\lambda \in K^\times$ let $M_\lambda$ be the $2$-dimensional $A$-module
defined by
$$
M_\lambda\df A \to M_2(K),\quad a \mapsto \BM0&0\\\lambda&0\EM,\quad
b \mapsto \BM0&0\\1&0\EM.
$$
We have
\begin{align*}
\hom_A(M_\lambda,M_\mu) &= 
\begin{cases}
2 & \text{if $\lambda = \mu$},
\\
1 & \text{otherwise},
\end{cases}
\\
\ext_A^1(M_\lambda,M_\mu) &= 
\begin{cases}
1 & \text{if $\lambda = \mu$},
\\
0 & \text{otherwise}.
\end{cases}
\end{align*}
Furthermore, we have
$M_\lambda \cong M_\mu$ if and only if $\lambda = \mu$.
Some easy computations show that
$\tau(M_\lambda) \cong M_\lambda$ and 
$\Omega(M_\lambda) \cong M_{-\lambda^{-1}}$.
Thus $\Omega^2(M_\lambda) \cong M_\lambda$.
One easily checks that
$$
Z := \overline{\bigcup_{\lambda \in K^\times} \cO_{M_\lambda}} \in \irr(A,2).
$$
(In fact, we have $\md(A,2) = Z$.)
Note that $Z$ does not contain a dense orbit.
It follows that
$\hom_A(Z,Z) = 1$, $\eend_A(Z) = 2$,
$\ext_A^1(Z,Z) = 0$, $\ext_A^1(Z) = 1$, $E(Z,Z) = 1$ and $E(Z) = 2$.
For $i \ge 2$ we get $\ext_A^i(Z,Z) = 0$ and
$$
\ext_A^i(Z) = 
\begin{cases}
1 & \text{if $i$ is odd},
\\ 
0 & \text{if $i$ is even}.
\end{cases}
$$
Here we used that $\Ext_A^i(M_\lambda,M_\mu) \cong \Ext_A^1(\Omega^{i-1}(M_\lambda),M_\mu)$.

%%%%%%%%%%%%%
\subsection{}
%%%%%%%%%%%%%
The following example is due to Calvin Pfeifer \cite{Pf22}. 
For $n \ge 2$ 
let $A = KQ/I$ where $Q$ is the quiver
$$
\xymatrix{
1 \ar@(ul,dl)_a & 2 \ar@<-0.5ex>[l]_{b_1}\ar@<0.5ex>[l]^{b_2} \ar@(ur,dr)[]^c
}
$$
and $I$ is generated by $\{ a^n,\, c^n,\, ab_i - b_ic \mid i = 1,2 \}$. 
This is the algebra $H(C,D,\Omega)$ from \cite{GLS17} for
$$
C = \BM 2&-2\\-2&2\EM,\quad D = \BM n&\\&n\EM,\quad
\Omega = \{ (1,2) \}. 
$$
Let $\bd = (n,n)$.
The locally free $A$-modules (in the sense of \cite{GLS17}) in $\md(A,\bd)$ form a non-empty open irreducible subset of $\md(A,2n)$. 
Let $Z \in \irr(A,2n)$ be the closure of this subset.
Then $Z$ is indecomposable, 
$\hom_A(Z,Z) = \ext_A^1(Z,Z) = E(Z,Z) = 0$ and 
$\eend_A(Z) = \ext_A^1(Z) = E(Z) = n$.

%%%%%%%%%%%%%
\subsection{}
%%%%%%%%%%%%%
Let $A = KQ$ be a finite-dimensional path algebra.
Then the irreducible components of $\md(A,d)$ are just the connected 
components of $\md(A,d)$.
Furthermore, each indecomposable $Z \in \irr(A)$ is a brick component,
see \cite[Proposition~1]{K82}.
Let $Q_0 = \{ 1,\ldots,n \}$ be the set of vertices of $Q$.
Thus $n = n(A)$.
Let $\bd = (d_1,\ldots,d_n)$ be an isotropic Schur root in the sense of
\cite{S92}, and let $d = d_1 + \cdots + d_n$.
 Let $Z = \md(A,\bd) \in \irr(A,d)$.
Then $\hom_A(Z,Z) = \ext_A^1(Z,Z) = E(Z,Z) = 0$ and 
$\eend_A(Z) = \ext_A^1(Z) = E(Z) = 1$.
If we assume additionally that $Q$ is connected and wild, then for all bricks $M \in Z$ we have $\tau(M) \not\cong M$. 
As a concrete example, let $Q$ be the quiver
$$
\xymatrix{
1 & \ar@<0.5ex>[l]\ar@<-0.5ex>[l] 2 & 3 \ar[l]
}
$$
and let $\bd = (1,2,1)$.
Then $Q$ is connected and wild, and $\bd$ is an isotropic Schur root.

%%%%%%%%%%%%%%%%%%%%%%%%%%%%%%%%%%%%%
%%%%%%%%%%%%%%%%%%%%%%%%%%%%%%%%%%%%%

\end{document}